\documentclass[preprint,12pt]{elsarticle}

\usepackage{amssymb}

\usepackage{mathptmx}      
\usepackage{amsmath}
\usepackage{amsfonts}

\def\Z{{\mathbb Z}}

\newcommand{\He}{\mathcal H}
\newcommand{\A}{\mathcal A}
\newcommand{\ug}{\stackrel{\rm def}{=}}

\newtheorem{teo}{Theorem}[section]
\newtheorem{lem}[teo]{Lemma}
\newtheorem{prop}[teo]{Proposition}
\newtheorem{cor}[teo]{Corollary}
\newtheorem{defi}[teo]{Definition}
\newproof{proof}{Proof}

\journal{}

\begin{document}

\begin{frontmatter}



\title{Kazhdan--Lusztig and $R$--polynomials of generalized Temperley--Lieb algebras\tnoteref{work}}
\tnotetext[work]{{This work is part of the author's doctoral dissertation, written under the direction of Prof. F. Brenti  at the University of Rome ``Tor Vergata''.}}




\author{Alfonso Pesiri\corref{dati}}
\ead{alfonso.pesiri@gmail.com}
\cortext[dati]{Dipartimento di Matematica, Universit\`a di Roma ``Tor Vergata'', Via della Ricerca Scientifica, I--00133, Roma, ITALY}

\begin{abstract}
We study two families of polynomials that play the same role, in the generalized Temperley--Lieb algebra of a Coxeter group, as the Kazhdan--Lusztig and $R$--polynomials in the Hecke algebra of the group. Our results include recursions, closed formulas, and other combinatorial properties for these polynomials. We focus mainly on non--branching Coxeter graphs.
\end{abstract}


\end{frontmatter}


\section*{Introduction} \label{intro}

The Temperley--Lieb algebra $TL(X)$ is a quotient of the Hecke algebra $\He(X)$ associated to a Coxeter group $W(X)$, $X$ being an arbitrary  Coxeter graph. It first appeared in \cite{tl-pcp}, in the context of statistical mechanics (see, \textrm{e.g.}, \cite{jo-pik}). The case $X=A$ was studied by Jones (see \cite{jo-har}) in connection to knot theory. For an arbitrary Coxeter graph, the Temperley--Lieb algebra was studied by Graham. More precisely, in \cite{gr-phd} Graham showed that $TL(X)$ is finite dimensional whenever $X$ is of type $A, B, D, E, F, H$ and $I$. If $X \not = A$ then $TL(X)$ is usually referred to as the generalized Temperley--Lieb algebra.
The algebra $TL(X)$ has many properties similar to the Hecke algebra $\He(X)$. In particular, in \cite{gl-cbh} Green and Losonczy show that $TL(X)$ always admits an \emph{IC basis} (see \cite{du-icb} and \cite{gl-cbh} for definitions and further details). These bases have properties similar to the well--known Kazhdan--Lusztig basis of the Hecke algebra $\He(X)$. Algebraic properties of these bases have been studied in \cite{gl-ppk} and \cite{gl-fck}. In this work, which is a continuation of the paper \cite{pes-cptl}, we investigate some combinatorial properties of them. More precisely, we look at the coefficients of the \emph{IC basis} of $TL(X)$ with respect to the standard basis, and obtain some recursive formulas for them. To do this, we find necessary to first study some auxiliary polynomials (which have no analogue in $\He(X)$, and which in some sense express the relationship between $\He(X)$ and $TL(X)$) which were first defined in \cite{gl-cbh}. As a consequence of these results we also obtain closed formulas for the polynomials expressing the inverse of an element of the standard basis as a linear combination of elements of the standard basis (or equivalently, for the coordinates of the canonical involution with respect to the standard basis). Most of our results hold for non--branching Coxeter graphs, although some hold in full generality. Our results emphasize the close relationship between Kazhdan--Lusztig and $R$--polynomials and their analogues in $TL(X)$.

The organization of the paper is as follows.   
In the next section we recall some generalities on the Hecke algebra, Kazhdan--Lusztig polynomials and the Kazhdan--Lusztig basis of $\He(X)$. Moreover, we recall the Temperley--Lieb algebra and the families of the polynomials $\{a_{x,w}\}$ and $\{L_{x,w}\}$ that we study in this work. In Sections \ref{sec: a}, \ref{sec: l} we prove our results on polynomials $\{a_{x,w}\}$ and $\{L_{x,w}\}$, which hold for all finite irreducible and affine non--branching Coxeter graphs $X$ such that $X \neq \widetilde{F_4}$, and we obtain an explicit formula for the polynomials $\{a_{x,w}\}$ in type $A$.

\section{Preliminaries}\label{ha}

In this section we recall some basic facts about Hecke algebras $\He(X)$ and Temperley--Lieb algebras $TL(X)$, $X$ being any Coxeter graph. Let $W(X)$ be the Coxeter group having $X$ as Coxeter graph and $S(X)$ as set of generators. Let $\A$ be the ring of Laurent polynomials $\mathbb{Z}[q^{\frac{1}{2}},q^{-\frac{1}{2}}]$. The Hecke algebra $\He(X)$ associated to $W(X)$ is an $\A$--algebra with linear basis $\{T_w:\,w\in W(X)\}$ (see, \textrm{e.g.}, \cite[\S 6.1]{bb-ccg} and \cite[\S 7]{Hum}). For all $w \in W(X)$ and $s \in S(X)$ the multiplication law is determined by
\begin{equation}\label{prod}
T_{w}T_{s} = \left\{ \begin{array}{ll} 
T_{ws}             & \mbox{if $\ell(ws)>\ell(w)$,} \\
q T_{ws}+(q-1)T_w  & \mbox{if $\ell(ws)<\ell(w)$,}
\end{array} \right.
\end{equation}
where $\ell$ denotes the usual length function of $W(X)$. We refer to $\{T_w:\,w\in W(X)\}$ as the $T$--basis for $\He(X)$. 

Let $e$ be the identity element of $W(X)$. One easily checks that $T_s^2=(q-1)T_s+qT_e$, being $T_e$ the identity element, and so $T_s^{-1}=q^{-1}(T_s-(q-1)T_e)$. It follows that all the elements $T_w$ are invertible, since, if $w=s_1 \cdots s_r$ and $\ell(w)=r$, then $T_w=T_{s_1} \cdots T_{s_r}$.
To express $T_w^{-1}$ as a linear combination of elements in the basis, one obtains the so--called \emph{$R$--polynomials}. For a proof of the following result we refer to \cite[\S 7.4]{Hum}.    

\begin{teo}\label{r-pol}
There is a unique family of polynomials $\{ R_{x,w}(q) \}_{x,w \in W(X)} \subseteq \mathbb{Z}[q]$ such that 
$$T_{w^{-1}}^{-1}=\varepsilon_wq^{-\ell(w)}\sum_{x\leq w}\varepsilon_x R_{x,w}(q)T_x,$$ 
and $R_{x,w}(q)=0$ if $x \not \leq w$, where $\varepsilon_x \ug (-1)^{\ell(x)}$. Furthermore, $R_{x,w}(q)=1$ if $x=w$. 
\end{teo}

Define a map $\iota:\He \rightarrow \He$ such that $\iota(T_w)=(T_{w^{-1}})^{-1}$, $\iota(q)=q^{-1}$ and extend by linear extension. We refer the reader to \cite[\S 7.7]{Hum} for the proof of the following result. 

\begin{prop}
The map $\iota$ is a ring homomorphism of order $2$ on $\He(X)$.
\end{prop}

In \cite{kl-rcg}, Kazhdan and Lusztig prove this basic theorem:

\begin{teo} \label{kl-pol}
There exists a unique basis $\{C_w:\, w \in W(X)\}$ for $\He(X)$ such that the following properties hold:
\begin{itemize}
\item[{\rm (i)}] $\iota(C_w)=C_w$,
\item[{\rm (ii)}] $C_w=\varepsilon_w q^{\frac{\ell(w)}{2}}\sum_{x\leq w}\varepsilon_x q^{-\ell(x)}P_{x,w}(q^{-1})T_x$,
\end{itemize}
where $\{P_{x,w}(q)\}\subseteq \mathbb{Z}[q]$, $P_{w,w}(q)=1$ and ${\rm deg}(P_{x,w}(q)) \leq \frac{1}{2}(\ell(w)-\ell(x)-1)$ if $x<w$. 
\end{teo}

The polynomials $\{P_{x,w}(q)\}_{x,w \in W(X)}$ are the so--called \emph{Kazhdan--Lusztig polynomials} of $W(X)$.
In \cite[\S 7.9]{Hum} it is shown that one can substitute the basis $\{C_w:\, w \in W(X)\}$ with the equivalent basis $\{C_w':\,w \in W(X)\}$, where 
\begin{equation}\label{C'w}
C_w'=q^{-\frac{\ell(w)}{2}}\sum_{x \leq w}P_{x,w}(q)T_x.
\end{equation}
For the rest of this paper we will refer to the latter basis as the \emph{Kazhdan--Lusztig basis} for $\He(X)$. 

Let $s_i, s_j \in S(X)$ and denote by $\langle s_i,s_j \rangle$ the parabolic subgroup of $W(X)$ generated by $s_i$ and $s_j$. Following \cite{gr-phd}, we consider the two--sided ideal $J(X)$ generated by all elements of $\He(X)$ of the form $$\sum_{w \in \langle s_i,s_j \rangle}T_w,$$ where $(s_i,s_j)$ runs over all pairs of non--commuting generators in $S(X)$ such that the order of $s_is_j$ is finite.   

\begin{defi}
The generalized Temperley--Lieb algebra is $TL(X)\stackrel{\rm def}{=}\He(X)/J(X)$.
\end{defi}

When $X$ is of type $A$, we refer to $TL(X)$ as the Temperley--Lieb algebra. In order to describe a basis for $TL(X)$, we recall the notion of a  \emph{fully commutative element} for $W(X)$ (see \cite{ste-fce}).

\begin{defi}
An element $w \in W(X)$ is fully commutative if any reduced expression for $w$ can be obtained from any other by applying Coxeter relations that involve only commuting generators. We let $$W_c(X) \stackrel{\rm def}{=} \{w \in W(X):\, w \mbox{ is a fully commutative element}\}.$$  
\end{defi}

If  $X=A_{n-1}$ then $W(X)=S_n$ (see \cite[Example 1.2.3]{bb-ccg}) and $W_c(A_{n-1})$ may be described as the set of elements of $W(A_{n-1})$ whose reduced expressions avoid substrings of the form $s_is_{i \pm 1}s_i$, for all $s_i \in S$ (see \cite[Proposition 1.1]{ste-fce}). Another description of $W_c(A_{n-1})$ may be given in terms of pattern avoidance: namely, in \cite[Theorem 2.1]{bjs-cps} Billey, Jockusch and Stanley show that $W_c(A_{n-1})$ coincides with the set of permutations avoiding the pattern $321$. Moreover $\vert W_c(A_{n-1}) \vert = C_n$, where $C_n=\frac{1}{n+1}\binom{2n}{n}$ denotes the $n$--th Catalan number (see  \cite[Proposition 3]{fan-haq} for further details). A similar characterization can be given in type $B$. If $X=B_n$ then $W_c(X)$ can be described as the group of signed permutations $S_n^B$ (see \cite[Example 1.2.4]{bb-ccg}). In \cite[Theorem 5.1]{ste-sca} Stembridge showed that the set of the signed permutations avoiding the patterns in $\{ \overline{12},\, 321,\, \overline{3}21,\, \overline{23}1,\, 2\overline{3}1 \}$ and $W_c(B_n)$ coincide. Moreover $\vert W_c(B_n) \vert = (n+2)C_n-1$ (see \cite[Proposition 5.9]{ste-sca}).

Let $t_w=\sigma(T_w)$, where $\sigma: \He \rightarrow \He /J$ is the canonical projection. A proof of the following can be found in \cite{gr-phd}.

\begin{teo}
$TL(X)$ admits an $\A$--basis of the form $\{t_w:\, w \in W_c(X)\}$.
\end{teo}

We call $\{t_w:\, w \in W_c(X)\}$ the \emph{$t$--basis} of $TL(X)$. By (\ref{prod}), it satisfies 
\begin{equation} \label{t-basmolt}
t_{w}t_{s} = \left\{ \begin{array}{ll} 
t_{ws}             & \mbox{if $\ell(ws)>\ell(w)$,} \\
q t_{ws}+(q-1)t_w  & \mbox{if $\ell(ws)<\ell(w)$.}
\end{array} \right.   
\end{equation}

Observe that if $ws \not \in W_c(X)$, then $t_{ws}$ can be expressed as linear combination of the $t$--basis elements by means of the following result (see \cite[Lemma 1.5]{gl-cbh}).

\begin{prop}\label{d-pol}
Let $w \in W(X)$. Then there exists a unique family of polynomials $\{D_{x,w}(q)\}_{x \in W_c(X)} \subseteq \mathbb{Z}[q]$ such that 
$$ t_w=\sum_{\substack{x \in W_c(X) \\ x \leq w}} D_{x,w}(q)t_x, $$
where $D_{w,w}(q)=1$ if $w \in W_c(X)$. Furthermore, $D_{x,w}(q)=0$ if $x \not \leq w$.
\end{prop} 

From the fact that the involution $\iota$ fixes the ideal $J(X)$ (see \cite[Lemma 1.4]{gl-cbh}), it follows that $\iota$ induces an involution on $TL(X)$, which we still denote by $\iota$, if there is no danger of confusion. More precisely, we have the following result.

\begin{prop}
The map $\iota$ is a ring homomorphism of order $2$ such that $\iota(t_w)=(t_{w^{-1}})^{-1}$ and $\iota(q)=q^{-1}$.
\end{prop}

To express the image of $t_w$ under $\iota$ as a linear combination of elements of the $t$--basis, one defines a new family of polynomials (see \cite[\S 2]{gl-cbh}).

\begin{prop}\label{a-pol}
Let $w \in W_c(X)$. Then there exists a unique family of polynomials $\{a_{y,w}(q)\} \subseteq \mathbb{Z}[q]$ such that 
\[ (t_{w^{-1}})^{-1}=q^{-\ell(w)}\sum_{\substack{y \in W_c(X) \\ y\leq w}}a_{y,w}(q)t_y, \]
where $a_{w,w}(q)=1$.
\end{prop}

The polynomials $\{a_{x,w}(q)\}$ associated to $TL(X)$ play the same role as the polynomials $\{R_{x,w}(q)\}$ associated to $\He(X)$. They both represent the coordinates of elements of the form $\iota(t_{w})$ (respectively, $\iota(T_{w})$) with respect to the $t$--basis (respectively, $T$--basis).

The generalized Temperley--Lieb algebra admits a basis $\{c_w:\, w \in W_c(X)\}$ which is analogous to the Kazhdan--Lusztig basis $\{C_w':\, w \in W(X)\}$ of $\He(X)$. The following is a restatement of \cite[Theorem 3.6]{gl-cbh}.

\begin{teo}\label{l-pol}
There exists a unique basis $\{c_w:\, w \in W_c\}$ of $TL(X)$ such that
\begin{itemize}
\item[{\rm (i)}]  $\iota(c_w)=c_w$, 
\item[{\rm (ii)}] $c_w=\sum_{\substack{x \in W_c \\ x \leq w}}q^{-\frac{\ell(x)}{2}}L_{x,w}(q^{-\frac{1}{2}})t_x$,
\end{itemize}
where $\{ L_{x,w}(q^{-\frac{1}{2}}) \} \subseteq q^{-\frac{1}{2}}\mathbb{Z}[q^{-\frac{1}{2}}]$, $L_{x,x}(q^{-\frac{1}{2}})=1$, and $L_{x,w}(q^{-\frac{1}{2}})=0$ if $x \not \leq w$.
\end{teo}

This basis is often called an \emph{IC basis} (see \cite[\S 2]{gl-cbh}).\\
Combining Theorem \ref{l-pol} with Proposition \ref{a-pol} we get
\begin{equation} \label{lpolapol}
L_{x,w}(q^{-\frac{1}{2}})=\sum_{y \in [x,w]_c}q^{\frac{\ell(x)-\ell(y)}{2}}a_{x,y}(q)L_{y,w}(q^{\frac{1}{2}}),
\end{equation}
for every $x,w \in W_c(X)$, with $[x,w]_c=\{y \in [x,w]:\, y \in W_c(X)\}$.

Comparing the definition of $c_w$ with that of $C_w'$, we notice that the polynomials $L_{x,w}(q^{-\frac{1}{2}})$ play the same role as $q^{\frac{\ell(x)-\ell(w)}{2}}P_{x,w}(q)$, where $P_{x,w}(q)$ are the Kazhdan--Lusztig polynomials defined in Theorem \ref{kl-pol}.
Since the Kazhdan--Lusztig basis and the IC basis are both $\iota$--invariant and since $\iota(J)=J$, it is natural to ask to what extent $\{\sigma(C_w'):\, w \in W(X) \}$ coincides with $\{c_w:\, w \in W_c(X)\}$. 

In particular, one may wonder whether the canonical projection $\sigma$ satisfies
\begin{equation} \label{Cw-0} 
\sigma(C_w') = \left\{ \begin{array}{ll} 
c_w  & \mbox{if $w \in W_c(X)$,} \\ 
0    & \mbox{if $w \not \in W_c(X)$.} 
\end{array} \right.
\end{equation}
If $X$ is a finite irreducible or affine Coxeter group, then relation (\ref{Cw-0}) holds if and only if $W_c(X)$ is a union of two–-sided Kazhdan-–Lusztig cells (see \cite[Lemma 2.4]{shi-fceII} and \cite[Theorem 2.2.3]{gl-fck}). On the other hand, in \cite[\S 3]{shi-fce} Shi shows that $W_c(X)$ is a union of two-–sided Kazhdan-–Lusztig cells if and only if $X$ is non--branching and $X \neq \widetilde{F_4}$. We sum up these properties in the following.

\begin{teo} \label{teocw0}
Let $X$ be a finite irreducible or affine Coxeter graph. Then, relation (\ref{Cw-0}) holds if and only if $X$ is non--branching and $X \neq \widetilde{F_4}$.
\end{teo}

\section{Combinatorial properties of polynomials $a_{x,w}$}\label{sec: a}

The first part of this section deals with the study of the $D$--polynomials defined in Proposition \ref{d-pol}. We recall a recurrence relation for $\{D_{x,w}\}_{x \in W_c(X), w \in W(X)}$, where $X$ denotes an arbitrary Coxeter graph. Then we focus on the Coxeter graphs satisfying equation (\ref{Cw-0}) and obtain an explicit formula for the $D$--polynomials indexed by elements which satisfy particular properties.\\
In the second part of the section we study the family of polynomials $\{a_{x,w}\}_{x,w \in W_c(X)}$, which express the involution $\iota$ in terms of the $t$--basis, as explained in Proposition \ref{a-pol}. First, we obtain a recurrence relation for $a_{x,w}$, $X$ being an arbitrary Coxeter graph. Then we derive an explicit formula for $a_{x,w}$, with $x,w \in W_c(X)$ satisfying particular properties and $X$ such that equation (\ref{Cw-0}) holds. 

We begin with the following recursion for the $D$--polynomials (see \cite[Theorem 3.1]{pes-cptl}).

\begin{teo} \label{p-gc}
Let $X$ be an arbitrary Coxeter graph. Let $w \not \in W_c(X)$ and $s \in S(X)$ be such that $ws \not \in W_c(X)$, with $ws<w$. Then, for all $x \in W_c(X),\,x \leq w$, we have 
$$D_{x,w}=\widetilde{D_{x,w}} + \sum_{\substack{y \in W_c(X),\, ys \not \in W_c(X) \\ ys>y}}D_{x,ys}D_{y,ws},$$
where
\[
\widetilde{D_{x,w}}\ug\left\{
\begin{array}{ll}
D_{xs,ws}+(q-1)D_{x,ws}       & \mbox{ if } xs<x,           \\
qD_{xs,ws}                    & \mbox{ if } x<xs \in W_c(X),   \\
0                             & \mbox{ if } x<xs \not \in W_c(X).
\end{array} \right.
\] 
\end{teo}

From here to the end of this section we will denote by $X$ a Coxeter graph satisfying (\ref{Cw-0}). Observe that $D_{x,w}=\delta_{x,w}$ if $x,w \in W_c(X)$.
\begin{lem} \label{lemain}
For all $x \in W_c(X)$ and $w \not \in W_c(X)$, we have $$\sum_{x \leq y \leq w} D_{x,y}P_{y,w}=0.$$
\end{lem}
A proof of the preceding lemma appears in \cite[Lemma 3.6]{pes-cptl}. It is worth noting that Lemma \ref{lemain} implies
\begin{equation} \label{lemaineq}
D_{x,w} = -P_{x,w} -\sum_{\substack{t \not \in W_c(X) \\ x < t < w }}D_{x,t}P_{t,w},
\end{equation}
for all $x \in W_c(X)$ and $w \not \in W_c(X)$ such that $x < w$.

\begin{lem} \label{ddelta}
Let $x \in W_c(X)$ be such that $xs \not \in W_c(X)$ and let $w \not \in W_c(X)$ be such that $w>ws \in W_c(X)$. Then $$D_{x,w}=-\delta_{x,ws}.$$
\end{lem}
\begin{proof}
We proceed by induction on $\ell(x,w)$. If $\ell(x,w)=1$, then $D_{x,w}=D_{x,xs}=-1=-\delta_{x,ws}$. Suppose $\ell(x,w)>1$. 
Recall that $P_{x,w}(q)=P_{xs,w}(q)$, for every $x \leq w$ such that $ws<w$ (see, \textrm{e.g.}, \cite[Proposition 5.1.8]{bb-ccg}). Then, from  (\ref{lemaineq}) we get 
\begin{eqnarray}
D_{x,w}   & = & -P_{x,w} -\sum_{\substack{t \not \in W_c(X) \\ x < t < w }}D_{x,t}P_{t,w}  \nonumber \\
          & = & -P_{x,w} -D_{x,xs}P_{xs,w}-\sum_{\substack{t \not \in W_c(X), t \neq xs \\ x < t < w }}D_{x,t}P_{t,w}  \nonumber \\
          & = & -P_{x,w} -D_{x,xs}P_{x,w}-\sum_{\substack{t \not \in W_c(X), t \neq xs \\ x < t < w }}D_{x,t}P_{t,w}  \nonumber \\
          & = & -\sum_{\substack{t \not \in W_c(X), t \neq xs \\ x < t < w }}D_{x,t}P_{t,w}  \nonumber \\
          & = & -\sum_{\substack{t>ts \in W_c(X), t \neq xs \\ x < t < w}}D_{x,t}P_{t,w}-\sum_{\substack{t>ts \not \in W_c(X) \\ x < t < w }}D_{x,t}P_{t,w}-\sum_{\substack{t<ts \\ x < t < w }}D_{x,t}P_{t,w}.  \nonumber
\end{eqnarray}
By induction hypothesis, the term $D_{x,t}$ in the first sum is equal to $-\delta_{x,ts}$, since $\ell(x,t)<\ell(x,w)$. Therefore, the first sum is zero. On the other hand, the second and the third sums can be written as 
\begin{equation} \label{iszero}
-\sum_{\substack{z \not \in W_c(X) \\ x<z<zs<w}} P_{z,w} \left(D_{x,zs}+D_{x,z} \right),
\end{equation}
since $t \not \in W_c(X),\, t<ts$ implies $ts \not \in W_c(X)$. To prove the statement we have to show that the term (\ref{iszero}) is zero. First, observe that $\ell(x,z)<\ell(x,w)$. Moreover, by Proposition \ref{p-gc} and by induction hypothesis, we achieve
\begin{equation}\label{dxzs0}
D_{x,zs}=\sum_{\substack{u \in W_c(X) \\ u<us \not \in W_c(X)}} D_{x,us}D_{u,z}=\sum_{\substack{u \in W_c(X) \\ u<us \not \in W_c(X)}} (-\delta_{x,u})D_{u,z}=-D_{x,z}.
\end{equation}
We conclude that $D_{x,zs}+D_{x,z}=0$, for all $z \not \in W_c(X)$ such that $x<z<zs<w$, and so the sum in (\ref{iszero}) is zero. \qed
\end{proof}

The next property for $D$--polynomials will be needed at the end of this section.
\begin{prop} \label{pro dsgn}
Let $w \in W(X)$. Then 
$$\sum_{\substack{x \in W_c(X) \\ x \leq w}}\varepsilon_x D_{x,w}=\varepsilon_w.$$
\end{prop}
\begin{proof}
We proceed by induction on $\ell(w)$. The proposition is trivial if $w \in W_c(X)$, which covers the case $\ell(w) \leq 2$. Suppose that $w \not \in W_c(X)$. Then, by (\ref{lemaineq}) we have
\begin{eqnarray}
\sum_{\substack{x \in W_c(X) \\ x \leq w}}\varepsilon_x D_{x,w} &=& \sum_{\substack{x \in W_c(X) \\ x < w}}\varepsilon_x(-P_{x,w}) + \sum_{\substack{x \in W_c(X) \\ x < w}}\varepsilon_x \left( -\sum_{\substack{t \not \in W_c(X) \\ x<t<w}} D_{x,t}P_{t,w} \right)   \nonumber \\
      &=& -\sum_{\substack{x \in W_c(X) \\ x<w}} \varepsilon_x P_{x,w} - \sum_{\substack{t \not \in W_c(X) \\ t<w}} P_{t,w} \left( \sum_{\substack{x \in W_c(X) \\ x<t}} \varepsilon_x D_{x,t} \right) \nonumber \\  
      &=& -\sum_{\substack{x \in W_c(X) \\ x<w}} \varepsilon_x P_{x,w} - \sum_{\substack{t \not \in W_c(X) \\ t<w}} P_{t,w} \varepsilon_t \nonumber \\  
      &=& -\sum_{x<w} \varepsilon_x P_{x,w}, \nonumber
\end{eqnarray} 
and the statement follows  from the fact that $\sum_{x \leq w} \varepsilon_x P_{x,w}=0$, for every $w \in W(X)\setminus \{e\}$ (see \cite[\S 5, Exercise 17]{bb-ccg}). \qed
\end{proof}

Now, let us turn our attention to the study of the polynomials $\{a_{x,w}\}_{x,w \in W_c(X)}$.
\begin{prop} \label{proarec}
Let $X$ be an arbitrary Coxeter graph. Let $w \in W_c(X)$ and $s \in S(X)$ be such that $w>ws \in W_c(X)$. Then, for all $x \in W_c(X),\,x \leq w$, we have 
$$a_{x,w}=\widetilde{a_{x,w}} + \sum_{\substack{y \in W_c(X),\, ys \not \in W_c(X) \\ ys>y}}D_{x,ys}a_{y,ws},$$
where
\[
\widetilde{a_{x,w}}\ug\left\{
\begin{array}{ll}
a_{xs,ws}                        & \mbox{ if } x>xs,           \\
qa_{xs,ws}+(1-q)a_{x,ws}         & \mbox{ if } x<xs \in W_c(X),   \\
(1-q)a_{x,ws}                    & \mbox{ if } x<xs \not \in W_c(X).
\end{array} \right.
\] 
\end{prop}
\begin{proof}
On the one hand, by Proposition \ref{a-pol}, we have 
\[ (t_{w^{-1}})^{-1}=q^{-\ell(w)}\sum_{\substack{y \in W_c(X) \\ y\leq w}}a_{y,w} t_y. \]
On the other hand, letting $v \stackrel{\rm def}{=} ws$, we get
\begin{eqnarray}
(t_{w^{-1}})^{-1} & = & (t_{v^{-1}})^{-1}(t_s)^{-1} \nonumber \\
                & = & q^{-\ell(v)}\sum_{\substack{y \in W_c(X) \\ y\leq v}}a_{y,v}t_y \cdot q^{-1}(t_s-(q-1)t_e)    \nonumber \\
                & = & q^{-\ell(w)} \left(\sum_{\substack{y \in W_c(X) \\ y\leq v}}a_{y,v}t_yt_s -(q-1)\sum_{\substack{y \in W_c(X) \\ y\leq v}}a_{y,v}t_y  \right)   \nonumber \\
                & = & q^{-\ell(w)} \left(\sum_{\substack{y \in W_c(X),\, ys \in W_c(X) \\ y\leq v,\,ys>y}} a_{y,v} t_{ys} + \sum_{\substack{y \in W_c(X),\, ys \not \in W_c(X) \\ y\leq v,\,ys>y}} a_{y,v} t_{ys} \right)  \nonumber \\
                &   & +\: q^{-\ell(w)} \left(\sum_{\substack{y \in W_c(X) \\ y\leq v,\,ys<y}} a_{y,v} t_{ys}-(q-1)\sum_{\substack{y \in W_c(X) \\ y\leq v}}a_{y,v}t_y \right) \nonumber \\
                & = & q^{-\ell(w)} \left(\sum_{\substack{y \in W_c(X),\, ys \in W_c(X) \\ y\leq v,\,ys>y}} a_{y,v} t_{ys} + \sum_{\substack{y \in W_c(X),\, ys \not \in W_c(X) \\ y\leq v,\,ys>y}} a_{y,v} \left(\sum_{\substack{z \in W_c(X) \\ z<sy}}D_{z,sy}t_z \right) \right) \nonumber \\
                &   & +\: q^{-\ell(w)} \left(\sum_{\substack{y \in W_c(X) \\ y\leq v,\,ys<y}} a_{y,v} (qt_{ys}+(q-1)t_y)-(q-1)\sum_{\substack{y \in W_c(X) \\ y\leq v}}a_{y,v}t_y \right) \nonumber \\                 
                & = & q^{-\ell(w)} \left(\sum_{\substack{z \in W_c(X) \\ z\leq v,\,z>zs}} a_{zs,v} t_{z} + \sum_{\substack{z \in W_c(X) \\ z<vs}} \left(\sum_{\substack{y \in W_c(X),\, ys \not \in W_c(X) \\ y \leq v,\,ys>y}} D_{z,sy}a_{y,v} \right)t_z \right)\nonumber \\
                &   & +\: q^{-\ell(w)} \left(\sum_{\substack{y \in W_c(X) \\ y\leq v,\,ys<y}} a_{y,v} qt_{ys}+ \sum_{\substack{z \in W_c(X) \\ z\leq v,\,z<zs}}  (q-1)t_{zs}-(q-1)\sum_{\substack{y \in W_c(X) \\ y\leq v}}a_{y,v}t_y \right), \nonumber
\end{eqnarray}
and the statement follows by extracting the coefficient of $t_x$. \qed
\end{proof}
From now on, we will assume $X$ to be any Coxeter graph satisfying equation (\ref{Cw-0}).
\begin{cor} \label{aqa}
Let $x,w \in W_c(X)$. If there exists $s \in S(X)$ such that $ws<w$ and $x<xs \not \in W_c(X)$, then $$a_{x,w}=-qa_{x,ws}.$$
\end{cor}
\begin{proof}
By Proposition \ref{proarec}, we have 
$$a_{x,w}=(1-q)a_{x,ws}+\sum_{\substack{y \in W_c(X),\, ys \not \in W_c(X) \\ ys>y}}D_{x,ys}a_{y,ws}.$$
On the other hand, by Lemma \ref{ddelta}, $D_{x,ys}=-\delta_{x,y}$. Therefore 
$$a_{x,w}=(1-q)a_{x,ws}-\sum_{\substack{y \in W_c(X),\, ys \not \in W_c(X) \\ ys>y}} \delta_{x,y} a_{y,ws}=(1-q)a_{x,ws}-a_{x,ws},$$
and the statement follows. \qed
\end{proof}

In the sequel we will need the following result (see \cite[Proposition 4.1]{pes-cptl}).
\begin{prop}\label{c-apol}
Let $x,w \in W_c(X)$ be such that $x \leq w$. Then
\begin{equation} \label{apolform}
a_{x,w}(q)=\varepsilon_x \varepsilon_w R_{x,w}(q)+\sum_{\substack{y \not \in W_c(X) \\ x < y < w}} \varepsilon_y \varepsilon_w R_{y,w}(q) D_{x,y}(q).
\end{equation}
\end{prop}

The recursion given in Corollary \ref{aqa} can sometimes be solved explicitly.
\begin{prop}
Let $s_is_{i+1}\cdots s_{i+k}s_{i-j} s_{i-j+1}\cdots s_i \cdots s_{i+k-1}$ be a reduced expression for $w \in W(A_n)$ and let $s_is_{i+1}\cdots s_{i+k}$ be a reduced expression for $x \in W(A_n)$, with $i \in [2,n],\,k \in [1,n-i],\,j \in [1,i-1]$. Then $$a_{x,w}(q)=(-q)^k(1-q)^j.$$
\end{prop}
\begin{proof} 
Observe that $x<xs_{i+h} \not \in W_c(X)$, for every $h \in [0,k-1]$ and that $ws_{i+k-1}<w$. By applying Corollary \ref{aqa} to the triple $(x,w,s_{i+k-1})$ we get $a_{x,w}=-qa_{x,ws_{i+k-1}}$. Repeat the same process with the triple $(x,ws_{i+k-1},s_{i+k-2})$, and so on. After $k$ iteration of the process we get $a_{x,w}(q)=(-q)^k a_{x,ws_{i+k-1}\cdots s_{i}}=(-q)^k a_{x,w'}(q)$, where we set $$w'=s_i s_{i+1}\cdots s_{i+k} s_{i-j} s_{i-j+1} \cdots s_{i-1}.$$ To conclude the proof, we will show that $a_{x,w'}(q)=(1-q)^j$. Observe that $[x,w'] \simeq B_{\ell(w')-\ell(x)}$ and so $R_{x,w'}(q)=(q-1)^{\ell(w')-\ell(x)}$ (see \cite[Corollary 4.10]{bre-cpkl}). On the other hand, Proposition \ref{c-apol} implies  $a_{x,w'}(q)=\varepsilon_x \varepsilon_{w'} R_{x,w'}(q)$, since $\{y \in [x,w']:\,y \not \in W_c(X)\}=\emptyset$. Therefore $a_{x,w'}(q)=\varepsilon_x \varepsilon_{w'}(q-1)^{\ell(w')-\ell(x)}=(1-q)^j$, as desired.  
\qed
\end{proof}

Next, we obtain a property for the polynomials $\{a_{x,w}\}$ which will be used in Section \ref{sec: l}.
\begin{prop} \label{apollo}
Let $w \in W_c(X)$. Then 
$$\sum_{\substack{x \in W_c(X) \\ x \leq w}} \varepsilon_x \varepsilon_w a_{x,w}=q^{\ell(w)}.$$
\end{prop}
\begin{proof}
First, it is a routine exercise to prove the following property: 
\begin{equation} \label{req}
\sum_{x \leq w} R_{x,w}=q^{\ell(w)},
\end{equation}
for every $w \in W(X)$.\\
By combining (\ref{apolform}) with Proposition \ref{pro dsgn} we get 
\begin{eqnarray}
\sum_{\substack{x \in W_c(X) \\ x \leq w}} \varepsilon_x \varepsilon_w a_{x,w}  
      &=& \sum_{\substack{x \in W_c(X) \\ x \leq w}} \varepsilon_x \varepsilon_w \left( \varepsilon_x \varepsilon_w R_{x,w}+\sum_{\substack{y \not \in W_c(X) \\ x< y <w}} \varepsilon_y \varepsilon_w R_{y,w} D_{x,y}  \right)   \nonumber \\
      &=& \sum_{\substack{x \in W_c(X) \\ x \leq w}} R_{x,w}  + \sum_{\substack{x \in W_c(X) \\ x \leq w}} \varepsilon_x \left( \sum_{\substack{y \not \in W_c(X) \\ x< y <w}} \varepsilon_y  R_{y,w} D_{x,y} \right)   \nonumber \\
      &=& \sum_{\substack{x \in W_c(X) \\ x \leq w}} R_{x,w}  + \sum_{\substack{y \not \in W_c(X) \\ y \leq w}} \varepsilon_y R_{y,w} \left( \sum_{\substack{x \in W_c(X) \\ x \leq y}} \varepsilon_x  D_{x,y}  \right)   \nonumber \\
      &=& \sum_{\substack{x \in W_c(X) \\ x \leq w}} R_{x,w}  + \sum_{\substack{y \not \in W_c(X) \\ y \leq w}} \varepsilon_y R_{y,w} \varepsilon_y   \nonumber \\
      &=& \sum_{x \leq w} R_{x,w} \nonumber
\end{eqnarray}
and the statement follows from (\ref{req}). \qed
\end{proof}

\section{Combinatorial properties of polynomials $L_{x,w}$}\label{sec: l}
In this section we study the polynomials $\{L_{x,w}(q^{-\frac{1}{2}})\}_{x,w \in W_c(X)}$, which play the same role, in $TL(X)$, as the Kazhdan--Lusztig polynomials in $\He(X)$. In particular, we derive a recursive formula for $L_{x,w}$ by means of some results in \cite{gre-gjt}. Then we obtain a recursion for $L_{x,w}$, with $x,w$ satisfying particular properties.\\ 
Throughout this section we will assume $X$ to be an arbitrary Coxeter graph satisfying (\ref{Cw-0}). We recall that $[x,w]_c$ denotes the set $\{y \in [x,w]: \, y \in W_c(X) \}$.

It is known that the terms of maximum possible degree in the polynomials $L_{x,w}$ and in the Kazhdan--Lusztig polynomials coincide (see \cite[Theorem 5.13]{gre-gjt}) . 
\begin{prop}\label{muu}
For $x,w \in W_c(X)$ let $M(x,w)$ be the coefficient of $q^{-\frac{1}{2}}$ in $L_{x,w}$ and let $\mu(x,w)$ be the coefficient of $q^{\frac{\ell(w)-\ell(x)-1}{2}}$ in $P_{x,w}$. Then $M(x,w) = \mu(x,w)$. 
\end{prop}
The product of two IC basis elements can be computed by means of the following formula (see \cite[Theorem 5.13]{gre-gjt}).
\begin{prop}\label{icprod}
Let $s \in S(X)$ and $w \in W_c(X)$. Then 
\[
c_sc_w= 
\begin{cases}
c_{sw}+\sum_{\substack{x \prec w  \\ sx<x}}\mu(x,w)c_x    & \mbox{if } \ell(sw)>\ell(w);  \\
(q^{\frac{1}{2}}+q^{-\frac{1}{2}})c_w                     & \mbox{otherwise},
\end{cases}   \nonumber
\]
where $c_x \ug 0$ for every $x \not \in W_c(X)$.
\end{prop}

\begin{cor} \label{txc} 
Let $s \in S(X)$ and $w \in W_c(X)$. Then 
\[
t_sc_w= 
\begin{cases}
-c_w+q^{\frac{1}{2}}\left(c_{sw} + \sum_{\substack{x \prec w  \\ sx<x}} \mu(x,w) c_x \right)    & \mbox{if } \ell(sw)>\ell(w);  \\
qc_w                                                                                            & \mbox{otherwise}.
\end{cases}   \nonumber
\]
\end{cor}

\begin{proof}
Observe that $t_s=q^{\frac{1}{2}}c_s-c_e$. So $t_sc_w=q^{\frac{1}{2}}c_sc_w-c_w$ and the statement follows by applying Proposition \ref{icprod}.\qed
\end{proof}

\begin{teo}
Let $x,w \in W_c(X)$ be such that $sx \in W_c(X)$ and $sw<w$. Then
\begin{eqnarray}
L_{x,w}(q^{-\frac{1}{2}}) & = & L_{sx,sw}(q^{-\frac{1}{2}})+q^{c-\frac{1}{2}}L_{x,sw}(q^{-\frac{1}{2}})-\sum_{\substack{sz<z \\ z \in [sx,sw]_c}} \mu(z,sw)L_{x,z}(q^{-\frac{1}{2}}) \nonumber \\
                          &   & +\: q^{-\frac{1}{2}}\sum_{\substack{sz \not \in W_c(X) \\z \in [x,w]_c}} q^{\frac{\ell(x)-\ell(z)}{2}}D_{x,sz}(q)L_{z,sw}(q^{-\frac{1}{2}}), \nonumber
\end{eqnarray}
where $c=1$ if $sx<x$ and $0$ otherwise. 
\end{teo}
\begin{proof}
Let $w=sv$. By Proposition \ref{icprod}, we have that
\begin{equation} \label{equazio}
c_w=c_{sv}=c_sc_v-\sum_{sz<z} \mu(z,sw)c_z.
\end{equation}
Recall that $c_s=q^{-\frac{1}{2}}(t_s+t_e)$. Hence
\begin{eqnarray}
c_sc_v & = & q^{-\frac{1}{2}} c_v + q^{-\frac{1}{2}} t_s c_v  \nonumber \\
       & = & q^{-\frac{1}{2}} c_v + \sum_{\substack{x \in W_c(X) \\ x \leq sw}}q^{-\frac{\ell(x)}{2}}L_{x,sw} t_s t_x  \nonumber \\
       & = & q^{-\frac{1}{2}} \left(c_v + \sum_{\substack{sx \in W_c(X) \\ x<sx}} q^{-\frac{\ell(x)}{2}}L_{x,sw} t_{sx}+\sum_{sx<x} q^{-\frac{\ell(x)}{2}}L_{x,sw} (qt_{sx}+(q-1)t_x) \right) \nonumber\\
       &   & +\: q^{-\frac{1}{2}}\left( \sum_{\substack{sx \not \in W_c(X) \\ x<sx}} q^{-\frac{\ell(x)}{2}}L_{x,sw} \left(\sum_{\substack{y \in W_c \\ y<sx}} D_{y,sx} t_y  \right)    \right) \nonumber \\
       & = & q^{-\frac{1}{2}} \left(c_v + \sum_{\substack{sx \in W_c(X) \\ x<sx}} q^{-\frac{\ell(x)}{2}}L_{x,sw} t_{sx}+\sum_{sx<x} q^{-\frac{\ell(x)}{2}}L_{x,sw} (qt_{sx}+(q-1)t_x)\right) \nonumber\\
       &   & +\: q^{-\frac{1}{2}} \left( \sum_{\substack{y \in W_c(X) \\ y \leq w}} \left(\sum_{\substack{sx \not \in W_c(X) \\ x<sx}} q^{-\frac{\ell(x)}{2}}D_{y,sx} L_{x,sw} \right) t_y    \right). \nonumber     
\end{eqnarray} 
Suppose that $su>u$ and extract the coefficient of $t_{su}$ on both sides of (\ref{equazio}). It follows that 
$$L_{su,w}=L_{u,sw}+q^{\frac{1}{2}}L_{su,sw}+\sum_{\substack{sz \not \in W_c(X) \\ z<sz}} q^{\frac{\ell(u)-\ell(z)}{2}}D_{su,sz}L_{z,sw}-\sum_{\substack{z \in [u,w]_c \\ sz<z}}\mu(z,sw) L_{su,z}.$$
Otherwise, if $su<u$ then
$$L_{su,w}=L_{u,sw}+q^{-\frac{1}{2}}L_{su,sw}+q^{-1}\sum_{\substack{sz \not \in W_c(X) \\ z<sz}} q^{\frac{\ell(u)-\ell(z)}{2}}D_{su,sz}L_{z,sw}-\sum_{\substack{z \in [u,w]_c \\ sz<z}}\mu(z,sw) L_{su,z}.$$
The statement follows by applying the substitution $x=su$.\qed
\end{proof}

In \cite[Theorem 5.1]{pes-cptl} the following result is proved.
\begin{teo}\label{lpoch}
Let $X$ be such that equation (\ref{Cw-0}) holds. For all elements $x,w \in W_c(X)$ such that $x < w$ we have
\begin{equation} \label{e-pol}
L_{x,w}=q^{\frac{\ell(x)-\ell(w)}{2}}\left(P_{x,w}+\sum_{\substack{y \not \in W_c(X) \\ x < y < w}} D_{x,y}P_{y,w}\right). 
\end{equation}
\end{teo}

\begin{lem} \label{lzw0}
Let $x,w \in W_c(X)$. If there exists $s \in S(X)$ such that $sw<w$ and $x<sx \not \in W_c(X)$, then $L_{x,w}=0$.
\end{lem}
\begin{proof}
By (\ref{e-pol}) we get 
\begin{eqnarray}
L_{x,w} &=& q^{\frac{\ell(x)-\ell(w)}{2}}\left(P_{x,w}+\sum_{\substack{y \not \in W_c(X) \\ x < y < w}} D_{x,y}P_{y,w}\right) \nonumber \\
        &=& q^{\frac{\ell(x)-\ell(w)}{2}}\left(P_{x,w}+D_{x,sx}P_{sx,w}+\sum_{\substack{y \not \in W_c(X), y \neq sx \\ x < y < w}} D_{x,y}P_{y,w}\right)\nonumber \\
        &=& q^{\frac{\ell(x)-\ell(w)}{2}} \left(\sum_{\substack{y \not \in W_c(X), y \neq sx \\ x < y < w}} D_{x,y}P_{y,w}\right). \label{star}
\end{eqnarray}
Denote by $(*)$ the expression in round brackets in (\ref{star}). Then $(*)$ is zero and the statement follows. In fact, by applying  relation (\ref{dxzs0}) and Lemma \ref{ddelta}, we get
\begin{eqnarray}
    (*) &=& \sum_{\substack{y \not \in W_c(X) \\ y < sy }} D_{x,y}P_{y,w}+\sum_{\substack{y \not \in W_c(X), y \neq sx \\ y>sy}} D_{x,y}P_{y,w}\nonumber \\
        &=& \sum_{\substack{y \not \in W_c(X) \\ y < sy \not \in W_c(X)}} D_{x,y}P_{y,w} + \sum_{\substack{y \not \in W_c(X) \\ y>sy \not \in W_c(X)}} D_{x,y}P_{y,w} + \sum_{\substack{y \not \in W_c(X), y \neq sx \\ y>sy \in W_c(X)}} D_{x,y}P_{y,w}\nonumber \\
        &=& \sum_{\substack{y \not \in W_c(X) \\ y>sy \not \in W_c(X)}} (\underbrace{D_{x,sy}+D_{x,y}}_{0})P_{y,w}+\sum_{\substack{y \not \in W_c(X), y \neq sx \\ y>sy \in W_c(X)}} D_{x,y}P_{y,w} \nonumber \\
        &=& \sum_{\substack{y \not \in W_c(X), y \neq sx \\ y>sy \in W_c(X)}} (-\delta_{x,sy})P_{y,w}=0, \nonumber    
\end{eqnarray}
as desired. \qed
\end{proof}

The next result is the anlogue of a well--known property of the Kazhdan--Lusztig polynomials (see, \textrm{e.g.}, \cite[Proposition 5.1.8]{bb-ccg}).
\begin{teo} \label{lrecurs}
Let $x,w \in W_c(X)$ be such that $x<w$. If there exists $s \in S(X)$ such that $sw<w$ and $x<sx \in W_c(X)$, then 
$$L_{x,w}=q^{-\frac{1}{2}}L_{sx,w}.$$  
\end{teo}
\begin{proof}
By Corollary \ref{txc} we get $t_sc_w=qc_w$, since $\ell(sw)<\ell(w)$ by hypothesis. Furthermore, by Theorem \ref{l-pol}, if $x<sx \in W_c(X)$ then $[t_{sx}](qc_w)=q\cdot q^{-\frac{\ell(sx)}{2}}L_{sx,w}$. On the other hand, Theorem \ref{l-pol} implies that
\begin{eqnarray}
t_sc_w & = & \sum_{\substack{x \in W_c(X) \\ x \leq w}}q^{-\frac{\ell(x)}{2}}L_{x,w} t_s t_x              \nonumber \\
       & = & \sum_{\substack{sx \in W_c(X) \\ sx>x}}q^{-\frac{\ell(x)}{2}}L_{x,w} t_{sx} + \sum_{\substack{sx \not \in W_c(X) \\ sx>x}}q^{-\frac{\ell(x)}{2}}L_{x,w} t_{sx}+               \nonumber \\
       &   & +\: \sum_{\substack{x \in W_c(X) \\ sx<x}}q^{-\frac{\ell(x)}{2}}L_{x,w} (qt_{sx} + (q-1) t_x)  \nonumber \\       
       & = & \sum_{\substack{sx \in W_c(X) \\ sx>x}}q^{-\frac{\ell(x)}{2}}L_{x,w} t_{sx} + \sum_{\substack{sx \not \in W_c(X) \\ sx>x}}q^{-\frac{\ell(x)}{2}}L_{x,w} \left( \sum_{\substack{y \in W_c(X) \\ y<sx}} D_{y,sx}t_y   \right) +               \nonumber \\
       &   & +\: q\sum_{\substack{sz \in W_c(X) \\ z<sz}}q^{-\frac{\ell(sz)}{2}}L_{sz,w} t_z + (q-1)\sum_{\substack{sz \in W_c(X) \\ z<sz}}q^{-\frac{\ell(sz)}{2}}L_{sz,w} t_{sz}  \nonumber \\
       & = & \sum_{\substack{sx \in W_c(X) \\ sx>x}}q^{-\frac{\ell(x)}{2}}L_{x,w} t_{sx} +  q^{\frac{1}{2}} q^{-\frac{\ell(x)}{2}}L_{sx,w} t_x + q^{\frac{1}{2}} q^{-\frac{\ell(x)}{2}}L_{sx,w} t_{sx} +    \label{lxsa} \\       
       &   & -\: q^{-\frac{1}{2}}q^{-\frac{\ell(x)}{2}}L_{sx,w} t_{sx} + \sum_{\substack{x \in W_c(X) \\ x \leq w}} \left(   \sum_{\substack{sz \not \in W_c(X) \\ z \in (x,w)_c}} q^{-\frac{\ell(z)}{2}}D_{x,sz} L_{z,w} \right)t_x.            \label{lxsb}
\end{eqnarray} 
By extracting the coefficient of $t_{sx}$ in (\ref{lxsa}) and (\ref{lxsb}) we obtain
\begin{eqnarray}
q^{\frac{1}{2}}q^{-\frac{\ell(x)}{2}}L_{sx,w} & = & q^{-\frac{\ell(x)}{2}}L_{x,w} + q^{\frac{1}{2}}q^{-\frac{\ell(x)}{2}} L_{sx,w} - q^{-\frac{1}{2}} q^{-\frac{\ell(x)}{2}}L_{sx,w} +          \nonumber \\
                                            &   & +\: \sum_{\substack{sz \not \in W_c(X) \\ z \in (x,w)_c}} q^{-\frac{\ell(z)}{2}}D_{sx,sz} L_{z,w},       \nonumber
\end{eqnarray}
that is 
$$L_{x,w}=q^{-\frac{1}{2}}L_{sx,w}-\sum_{\substack{sz \not \in W_c(X) \\ z \in (x,w)_c}}q^{\frac{\ell(x)-\ell(z)}{2}} D_{sx,sz} L_{z,w}.$$
Observe that Lemma \ref{lzw0} implies $L_{z,w}=0$, since $z<sz \not \in W_c(X)$, and the statement follows.\qed
\end{proof}

We conclude this section with two results inspired by similar properties for the Kazhdan--Lusztig polynomials (see, \textrm{e.g.}, \cite[\S5, Exercises 16, 17]{bb-ccg}). 
\begin{prop} \label{prop-l}
Let $w \in W_c(X)$ and define $$F_w(q^{-\frac{1}{2}})\stackrel{\rm def}{=}\sum_{\substack{x \in W_c(X) \\ x \leq w}}\varepsilon_x q^{-\frac{\ell(x)}{2}}L_{x,w}(q^{-\frac{1}{2}}).$$  Then $F_w(q^{-\frac{1}{2}})=\delta_{e,w}$.
\end{prop}
\begin{proof}
The case $w=e$ is trivial. Suppose $w \not = e$. Combining (\ref{lpolapol}) with Proposition \ref{apollo} we have
\begin{eqnarray}
F_w(q^{-\frac{1}{2}}) &=& \sum_{\substack{u \in W_c(X) \\ u \leq w}}\varepsilon_u q^{-\frac{\ell(u)}{2}} \left( \sum_{\substack{x \in W_c(X) \\ u \leq x \leq w}} q^{\frac{\ell(u)-\ell(x)}{2}} a_{u,x}(q) L_{x,w}(q^{\frac{1}{2}})  \right) \nonumber \\  
                      &=& \sum_{\substack{x \in W_c(X) \\ x \leq w}} \left( \sum_{\substack{u \in W_c(X) \\ u \leq x}} \varepsilon_u q^{-\frac{\ell(x)}{2}} a_{u,x}(q) L_{x,w}(q^{\frac{1}{2}})  \right) \nonumber \\
                      &=& \sum_{\substack{x \in W_c(X) \\ x \leq w}} \varepsilon_x q^{-\frac{\ell(x)}{2}} L_{x,w}(q^{\frac{1}{2}}) \left( \sum_{\substack{u \in W_c(X) \\ u \leq x}} \varepsilon_x \varepsilon_u a_{u,x}(q) \right) \nonumber \\
                      &=& \sum_{\substack{x \in W_c(X) \\ x \leq w}} \varepsilon_x q^{-\frac{\ell(x)}{2}} L_{x,w}(q^{\frac{1}{2}}) q^{\ell(x)} \nonumber \\
                      &=& \sum_{\substack{x \in W_c(X) \\ x \leq w}} \varepsilon_x q^{\frac{\ell(x)}{2}} L_{x,w}(q^{\frac{1}{2}})  \nonumber \\ 
                      &=& F_w(q^{\frac{1}{2}}). \nonumber 
\end{eqnarray}
This implies that $F_w(q^{-\frac{1}{2}})$ is constant. On the other hand, the constant term in $F_w(q^{-\frac{1}{2}})$ is zero since $L_{x,w} \in q^{-\frac{1}{2}}\Z[q^{-\frac{1}{2}}]$ by Theorem \ref{l-pol}, and the statement follows.
\qed
\end{proof}

\newpage

\section*{Acknowledgements}
I would like to thank Prof. Francesco Brenti for introducing me to this topic and for many useful conversations.












\end{document}